\documentclass[10pt,leqno]{article}

\usepackage{amsmath,amssymb,amsthm,mathrsfs,dsfont}

\usepackage[margin=3cm]{geometry} 

\usepackage{titlesec,hyperref}

\usepackage{color}

\usepackage{fancyhdr}
\titleformat{\subsection}{\it}{\thesubsection.\enspace}{1pt}{}

\newtheorem{theo}{Theorem}[section]
\newtheorem{lemm}[theo]{Lemma}

\newtheorem{prop}[theo]{Proposition}
\newtheorem{rema}[theo]{Remark}
\numberwithin{equation}{section}

\allowdisplaybreaks 

\newcommand\lm{{\lesssim}}

\begin{document}
\title{Large time behavior of global strong solutions to the 2-D compressible FENE dumbbell model
\hspace{-4mm}
}

\author{ Zhaonan $\mbox{Luo}^1$ \footnote{email: 1411919168@qq.com},\quad
Wei $\mbox{Luo}^1$\footnote{E-mail:  luowei23@mail2.sysu.edu.cn} \quad and\quad
 Zhaoyang $\mbox{Yin}^{1,2}$\footnote{E-mail: mcsyzy@mail.sysu.edu.cn}\\
 $^1\mbox{Department}$ of Mathematics,
Sun Yat-sen University, Guangzhou 510275, China\\
$^2\mbox{Faculty}$ of Information Technology,\\ Macau University of Science and Technology, Macau, China}

\date{}
\maketitle
\hrule

\begin{abstract}
In this paper, we mainly study large time behavior of the strong solutions to the 2-D compressible finite extensible nonlinear elastic (FENE) dumbbell model. The Fourier splitting method yields that the $L^2$ decay rate is $\ln^{-l}(e+t)$ for any $l\in\mathbb{N}$. By virtue of the time weighted energy estimate, we can improve the decay rate to $(1+t)^{-\frac{1}{4}}$. Under the low-frequency condition and by the Littlewood-Paley theory, we show that the solutions belong to some Besov space with negative index and obtain the optimal $L^2$ decay rate.  \\
\vspace*{5pt}

\vspace*{5pt}
\end{abstract}

\vspace*{10pt}

\tableofcontents

\section{Introduction}
In this paper we are concerned with the compressible finite extensible nonlinear elastic (FENE) dumbbell model \cite{Bird1977,Doi1988}:
\begin{align}\label{eq0}
\left\{
\begin{array}{ll}
\varrho_t+div(\varrho u)=0 , \\[1ex]
(\varrho u)_t+div(\varrho u\otimes u)-div\Sigma{(u)}+\nabla_x{P(\varrho)}=div~\tau, \\[1ex]
\psi_t+u\cdot\nabla\psi=div_{R}[- \sigma(u)\cdot{R}\psi+\nabla_{R}\psi+\nabla_{R}\mathcal{U}\psi],  \\[1ex]
\tau_{ij}=\int_{B}(R_{i}\nabla_{j}\mathcal{U})\psi dR, \\[1ex]
\varrho|_{t=0}=\varrho_0,~~u|_{t=0}=u_0,~~\psi|_{t=0}=\psi_0, \\[1ex]
(\nabla_{R}\psi+\nabla_{R}\mathcal{U}\psi)\cdot{n}=0 ~~~~ \text{on} ~~~~ \partial B(0,R_{0}) .\\[1ex]
\end{array}
\right.
\end{align}
In \eqref{eq0}, $\varrho(t,x)$ stands for the density of the solvent, $u(t,x)$ denotes the velocity of the polymeric liquid and $\psi(t,x,R)$ represents the distribution function for the internal configuration. The notation $\Sigma{(u)}=\mu(\nabla u+\nabla^{T} u)+\mu'div~u\cdot Id$ stands for the stress tensor, with the viscosity coefficients $\mu$ and $\mu'$ satisfying $\mu>0$ and $2\mu+\mu'>0$. The pressure satisfies $P(\varrho)=\varrho^\gamma$ with $\gamma\geq1$. Moreover, $\tau$ stands for an additional stress tensor. $\sigma(u)$ is the drag term. In the general case, $\sigma(u)=\nabla u$.

The compressible FENE dumbbell model describes the system coupling fluids and polymers. A polymer is described as an "elastic dumbbell" consisting of two "beads" joined by a spring that can be modeled by $R$. The finite extensibility of the polymers means that $R$ is bounded in ball $ B=B(0,R_{0})$. Let $\mathcal{U}(R)=-k\log(1-(\frac{|R|}{|R_{0}|})^{2})$ represents the potential with $k>0$. The polymer particles are described by $\psi(t,x,R)$ satisfies $\int_{B} \psi(t,x,R)dR =1$. At the level of liquid, the system couples compressible Navier-Stokes (CNS) equations with a Fokker-Planck equation. This is a micro-macro model (For more details, one can refer to $\cite{Masmoudi2008}$ and $\cite{Masmoudi2013}$).

%
%
%


In this paper, we take $R_{0}=1$ and $x\in\mathbb{R}^2$.
Notice that $(\varrho,u,\psi)$ with $\varrho=1$, $u=0$ and $$\psi_{\infty}(R)=\frac{e^{-\mathcal{U}(R)}}{\int_{B}e^{-\mathcal{U}(R)}dR}=\frac{(1-|R|^2)^k}{\int_{B}(1-|R|^2)^kdR},$$
is a trivial solution of \eqref{eq0}. Taking the perturbations near the global equilibrium:
\begin{align*}
\rho=\varrho-1,~~u=u,~~g=\frac {\psi-\psi_\infty} {\psi_\infty},
\end{align*}
then we can rewrite \eqref{eq0} as the following system:
\begin{align}\label{eq1}
\left\{
\begin{array}{ll}
\rho_t+div~u(1+\rho)=-u\cdot\nabla\rho , \\[1ex]
u_t-\frac 1 {1+\rho} div\Sigma{(u)}+\frac {P'(1+\rho)} {1+\rho} \nabla\rho=-u\cdot\nabla u+\frac 1 {1+\rho} div~\tau, \\[1ex]
g_t+\mathcal{L}g=-u\cdot\nabla g-\frac 1 {\psi_\infty}\nabla_R\cdot(\nabla uRg\psi_\infty)-div~u-\nabla u R\nabla_{R}\mathcal{U},  \\[1ex]
\tau_{ij}(g)=\int_{B}(R_{i}\nabla_{Rj}\mathcal{U})g\psi_\infty dR, \\[1ex]
\rho|_{t=0}=\rho_0,~~u|_{t=0}=u_0,~~g|_{t=0}=g_0, \\[1ex]
\psi_\infty\nabla_{R}g\cdot{n}=0 ~~~~ \text{on} ~~~~ \partial B(0,1) ,\\[1ex]
\end{array}
\right.
\end{align}
where $\mathcal{L}g=-\frac 1 {\psi_\infty}\nabla_R\cdot(\psi_\infty\nabla_{R}g)$.

{\bf Remark.} As in the reference \cite{Masmoudi2013}, one can deduce that $\psi=0$ on $\partial B(0,1)$.

\subsection{Short reviews for the incompressible FENE model}
Let us review some mathematical results about the incompressible FENE model. At first, M. Renardy \cite{Renardy} established the local well-posedness in Sobolev spaces with potential $\mathcal{U}(R)=(1-|R|^2)^{1-\sigma}$ for $\sigma>1$. Later, B. Jourdain, T. Leli\`{e}vre, and
C. Le Bris \cite{Jourdain} proved local existence of a stochastic differential equation with potential $\mathcal{U}(R)=-k\log(1-|R|^{2})$ in the case $k>3$ for a Couette flow. H. Zhang and P. Zhang \cite{Zhang-H} proved local well-posedness of (1.4) with $d=3$ in weighted Sobolev spaces. For the co-rotation case, F. Lin, P. Zhang, and Z. Zhang \cite{F.Lin} obtain a global existence results with $d=2$ and $k > 6$. If the initial data is perturbation around equilibrium, N. Masmoudi \cite{Masmoudi2008} proved global well-posedness of (1.4) for $k>0$. In the co-rotation case with $d=2$, he \cite{Masmoudi2008} obtained a global result for $k>0$ without any small conditions. A. V. Busuioc, I. S. Ciuperca, D. Iftimie and L. I. Palade \cite{Busuioc} obtain a global existence result with only the small condition on $\psi_0$. The global existence of weak solutions in $L^2$ was proved recently by N. Masmoudi \cite{Masmoudi2013} under some entropy conditions.
Recently, M. Schonbek \cite{Schonbek} studied the $L^2$ decay of the velocity for the co-rotation
FENE dumbbell model, and obtained the
decay rate $(1+t)^{-\frac{d}{4}+\frac{1}{2}}$, $d\geq 2$ with $u_0\in L^1$.
Moreover, she conjectured that the sharp decay rate should be $(1+t)^{-\frac{d}{4}}$,~$d\geq 2$.
However, she failed to get it because she could not use the bootstrap argument as in \cite{Schonbek1985} due to the
additional stress tensor.  Recently, W. Luo and Z. Yin \cite{Luo-Yin} improved Schonbek's result
and showed that the decay rate is $(1+t)^{-\frac{d}{4}}$ with $d\geq 3$ and $\ln^{-l}(1+t)$ with $d=2$ for any $l\in\mathbb{N^+}$.
This result shows that M. Schonbek's conjecture is true when $d \geq3$. More recently, W. Luo and Z. Yin \cite{Luo-Yin2} improved the decay rate to $(1+t)^{-\frac{d}{4}}$ with $d=2$. In \cite{Jia-Ye}, Y. Jia and H. Ye also prove the optimal time decay rate via a different approach.

\subsection{Short reviews for the CNS equations}
Taking $\psi\equiv 0$, the system \eqref{eq0} reduce to the CNS equations. In order to study about the \eqref{eq0}, we cite some reference about the CNS equations. The first local existence and uniqueness results were obtained by J. Nash \cite{miaocompressns23} for smooth initial data without vacuum. Later on, A. Matsumura and T. Nishida \cite{Matsumura} proved the global well-posedness for smooth data close to equilibrium with $d=3$. In \cite{miaocompressns18}, A. V. Kazhikhov and V. V. Shelukhin established the first global existence result with large data in one dimensional space under some suitable condition on $\mu$ and $\lambda$. If $\mu$ is constant and $\lambda(\rho)=b\rho^\beta$, X. Huang and J. Li \cite{Huang} obtained a global existence and uniqueness result for large initial data in two dimensional space(See also \cite{miaocompressns26}). In \cite{miaocompressns25}, X. Huang, J. Li, and Z. Xin proved the global well-posedness with vacuum. The Blow-up phenomenons were studied by Z. Xin et al
in \cite{Xin98,miaocompressns28,miaocompressns27}. Concerning the global existence of weak solutions for the large initial data, we may refer to \cite{miaocompressns2,miaocompressns3,miaocompressns21,Vasseur}.

To catch the scaling invariance property of the CNS equations. R. Danchin introduced the "critical spaces" in his series papers \cite{miaocompressns11,miaocompressns12,miaocompressns14,miaocompressns15,miaocompressns155} and obtained several important existence and uniqueness results. Recently, Q. Chen, C. Miao and Z. Zhang \cite{miao} proved the local existence and uniqueness in critical homogeneous Besov spaces. The ill-posedness results was obtained in \cite{miaoill-posedness}. In \cite{miaocompressns24}, L. He, J. Huang and C. Wang proved the global stability with $d=3$ i.e. for any perturbed solutions will remain close to the reference solutions if initially they are close to another one.

The large time behaviour of the global solutions $(\rho,u)$ to the CNS equations with $d=3$ was firstly proved by A. Matsumura and T. Nishida in \cite{Matsumura}. Recently, H. Li and T. Zhang \cite{Li2011Large} obtain the optimal time decay rate for the CNS equations with $d=3$ by spectrum analysis in Sobolev spaces. R. Danchin  and J. Xu \cite{2016Optimal} studied about the large time behaviour in the critical Besov space with $d\geq 2$. J. Xu \cite{Xu2019} obtained the optimal time decay rate with a small low-frequency assumption in some Besov spaces with negative index. More recently, Z. Xin  and J. Xu \cite{2018Optimal} studied about the large time behaviour and removed the smallness of low frequencies.

\subsection{Main results}

The well-posedness for the system \eqref{eq0} was established by J. Ning, Y. Liu and T. Zhang \cite{2017Global}. They proved the global well-posedness for \eqref{eq0} if the initial data is close to the equilibrium and $R\in\mathbb{R}^3$. For the compressible FENE system, Z. Luo, W. Luo and Z. Yin \cite{2021Luo} proved the global well-posedness results with $d\geq2$ and studied large time behaviour with $d\geq3$.

Recently, N. Masmoudi \cite{2016Equations} is concerning with the long time behavior for polymeric models. Large time behaviour of the compressible FENE system with $d\geq3$ has been studied in \cite{2021Global} and \cite{2021Luo}. To our best knowledge, for $d=2$, large time behaviour for the compressible FENE system \eqref{eq1} has not been studied yet. This problem is interesting and more difficult than the case with $d\geq 3$. In this paper, we will study about the optimal time decay rate for global strong solutions in $L^2$. The proof is based on the Littlewood-Paley decomposition theory and the Fourier splitting method. Motivated by \cite{He2009} and \cite{2018Global}, we deal with the linear term $div~\tau$ by virtue of the cancellation relation in Fourier space. Firstly we can prove the logarithmic decay rate for the velocity via the method in \cite{Luo-Yin}. By virtue of the time weighted energy
estimate and the logarithmic decay rate, then we improve the time decay rate to $(1+t)^{-\frac{1}{4}}$ for the velocity in $L^2$. The main difficult to get optimal time decay rate is that we can not get any information of $u$ in $L^1$ from \eqref{eq1}. Fortunately, similar to \cite{Tong2017The}, we can prove $\|u\|_{L^\infty(0,\infty; \dot{B}^{-\frac 1 2}_{2,\infty})}\leq C$ from \eqref{eq1} by using the time decay rate $(1+t)^{-\frac{1}{4}}$. Then we improve the time decay rate to $(1+t)^{-\frac{5}{16}}$ for the velocity in $L^2$ by the Littlewood-Paley decomposition theory and the Fourier splitting method. We obtain a slightly weaker conclusion $\|u\|_{L^\infty(0,\infty; \dot{B}^{-1}_{2,\infty})}\leq C$ from \eqref{eq1} by using the time decay rate $(1+t)^{-\frac{5}{16}}$. Finally, we obtain optimal time decay rate for the velocity in $L^2$ by the Littlewood-Paley decomposition theory and the standard Fourier splitting method.

Our main result can be stated as follows.
\begin{theo}[Large time behaviour]\label{th2}
Let $d=2$. Let $(\rho,u,g)$ be a global strong solution of \eqref{eq1} with the initial data $(\rho_0,u_0,g_0)$ under the condition in Theorem \ref{th1}. In addition, if $(\rho_0,u_0)\in \dot{B}^{-1}_{2,\infty}\times \dot{B}^{-1}_{2,\infty}$ and $g_0\in \dot{B}^{-1}_{2,\infty}(\mathcal{L}^2)$, then there exists a constant $C$ such that
\begin{align}\label{decay0}
\|\rho\|_{L^2}+\|u\|_{L^2}\leq C(1+t)^{-\frac 1 2}
\end{align}
and
\begin{align}
\|g\|_{L^2(\mathcal{L}^2)}\leq C(1+t)^{-1}.
\end{align}
\end{theo}

\begin{rema}
Taking $\psi \equiv 0$ and combining with the result in \cite{Xu2019}, we can see that the $L^2$ decay rate for $(\rho,u)$ obtained in Theorem \ref{th2} is optimal.
\end{rema}

\begin{rema}
In previous papers, researchers usually add the condition $(\rho_0,u_0)\in L^1\times  L^1$ to obtain the optimal time decay rate. Since $L^1\hookrightarrow \dot{B}^{-1}_{2,\infty}$, it follows that our condition is weaker and the results still hold true for $(\rho_0, u_0)\in L^1\times  L^1$. Moreover, the assumption can be replaced with a weaker assumption $\sup_{j\leq j_0}2^{-j}\|\dot{\Delta}_j (\rho_0,u_0,g_0)\|_{L^2\times L^2\times L^2(\mathcal{L}^2)} <\infty$, for any $j_0\in \mathbb{Z}$.
\end{rema}

The paper is organized as follows. In Section 2 we introduce some notations and  give some preliminaries which will be used in the sequel. In Section 3 we prove the $L^2$ decay of solutions to the 2-D compressible FENE model
by using the Fourier splitting method, the Littlewood-Paley decomposition theory and the bootstrap argument.

\section{Preliminaries}
In this section we introduce some notations and useful lemmas which will be used in the sequel.

For $p\geq1$, we denote by $\mathcal{L}^{p}$ the space
$$\mathcal{L}^{p}=\big\{f \big|\|f\|^{p}_{\mathcal{L}^{p}}=\int_{B} \psi_{\infty}|f|^{p}dR<\infty\big\}.$$
We use the notation $L^{p}_{x}(\mathcal{L}^{q})$ to denote $L^{p}[\mathbb{R}^{2};\mathcal{L}^{q}]:$
$$L^{p}_{x}(\mathcal{L}^{q})=\big\{f \big|\|f\|_{L^{p}_{x}(\mathcal{L}^{q})}=(\int_{\mathbb{R}^{2}}(\int_{B} \psi_{\infty}|f|^{q}dR)^{\frac{p}{q}}dx)^{\frac{1}{p}}<\infty\big\}.$$

The symbol $\widehat{f}=\mathcal{F}(f)$ represents the Fourier transform of $f$.
Let $\Lambda^s f=\mathcal{F}^{-1}(|\xi|^s \widehat{f})$.
If $s\geq0$, we denote by $H^{s}(\mathcal{L}^{2})$ the space
$$H^{s}(\mathcal{L}^{2})=\{f\big| \|f\|^2_{H^{s}(\mathcal{L}^{2})}=\int_{\mathbb{R}^{2}}\int_B(|f|^2+|\Lambda^s f|^2)\psi_\infty dRdx<\infty\}.$$

Sometimes we write $f\lm g$ instead of $f\leq Cg$, where $C$ is a constant. We agree that $\nabla$ stands for $\nabla_x$ and $div$ stands for $div_x$.

For the convenience of readers, we recall the Littlewood-Paley decomposition theory and and Besov spaces.
\begin{prop}\cite{Bahouri2011}\label{pro0}
Let $\mathcal{C}$ be the annulus $\{\xi\in\mathbb{R}^2:\frac 3 4\leq|\xi|\leq\frac 8 3\}$. There exist radial function $\varphi$, valued in the interval $[0,1]$, belonging respectively to $\mathcal{D}(\mathcal{C})$, and such that
$$ \forall\xi\in\mathbb{R}^2\backslash\{0\},\ \sum_{j\in\mathbb{Z}}\varphi(2^{-j}\xi)=1, $$
$$ |j-j'|\geq 2\Rightarrow\mathrm{Supp}\ \varphi(2^{-j}\cdot)\cap \mathrm{Supp}\ \varphi(2^{-j'}\cdot)=\emptyset. $$
Further, we have
$$ \forall\xi\in\mathbb{R}^2\backslash\{0\},\ \frac 1 2\leq\sum_{j\in\mathbb{Z}}\varphi^2(2^{-j}\xi)\leq 1. $$
\end{prop}

Let $u$ be a tempered distribution in $\mathcal{S}'_h(\mathbb{R}^2)$. For all $j\in\mathbb{Z}$, define
$$\dot{\Delta}_j u=\mathcal{F}^{-1}(\varphi(2^{-j}\cdot)\mathcal{F}u).$$
Then the Littlewood-Paley decomposition is given as follows:
$$ u=\sum_{j\in\mathbb{Z}}\dot{\Delta}_j u \quad \text{in}\ \mathcal{S}'(\mathbb{R}^2). $$
Let $s\in\mathbb{R},\ 1\leq p,r\leq\infty.$ The homogeneous Besov space $\dot{B}^s_{p,r}$ and $\dot{B}^s_{p,r}(\mathcal{L}^q)$ is defined by
$$ \dot{B}^s_{p,r}=\{u\in \mathcal{S}'_h:\|u\|_{\dot{B}^s_{p,r}}=\Big\|(2^{js}\|\dot{\Delta}_j u\|_{L^p})_j \Big\|_{l^r(\mathbb{Z})}<\infty\}, $$
$$ \dot{B}^s_{p,r}(\mathcal{L}^q)=\{\phi\in \mathcal{S}'_h:\|\phi\|_{\dot{B}^s_{p,r}(\mathcal{L}^q)}=\Big\|(2^{js}\|\dot{\Delta}_j \phi\|_{L_{x}^{p}(\mathcal{L}^q)})_j \Big\|_{l^r(\mathbb{Z})}<\infty\}.$$

We introduce the following results about inclusions between Lesbesgue and Besov spaces
\begin{lemm}\cite{Bahouri2011}\label{Lemma}
Let $1\leq p\leq 2$ and $d=2$. Then we have
$$L^{p}\hookrightarrow \dot{B}^{1-\frac 2 p}_{2,\infty}.$$
\end{lemm}

The following lemma is the Gagliardo-Nirenberg inequality of Sobolev type.
\begin{lemm}\cite{1959On}\label{Lemma0}
Let $d\geq2,~p\in[2,+\infty)$ and $0\leq s,s_1\leq s_2$, then there exists a constant $C$ such that
 $$\|\Lambda^{s}f\|_{L^{p}}\leq C \|\Lambda^{s_1}f\|^{1-\theta}_{L^{2}}\|\Lambda^{s_2} f\|^{\theta}_{L^{2}},$$
where $0\leq\theta\leq1$ and $\theta$ satisfy
$$ s+d(\frac 1 2 -\frac 1 p)=s_1 (1-\theta)+\theta s_2.$$
Note that we require that $0<\theta<1$, $0\leq s_1\leq s$, when $p=\infty$.
\end{lemm}

The following lemma allows us to estimate the extra stress tensor $\tau$.
\begin{lemm}\cite{Masmoudi2008}\label{Lemma1}
 If $\int_B g\psi_\infty dR=0$, then there exists a constant $C$ such that
 $$\|g\|_{\mathcal{L}^{2}}\leq C \|\nabla _{R} g\|_{\mathcal{L}^{2}}.$$
\end{lemm}

\begin{lemm}\label{Lemma2}
\cite{Masmoudi2008} For all $\delta>0$, there exists a constant $C_{\delta}$ such that
$$|\tau(g)|^2\leq\delta\|\nabla _{R}g\|^2_{\mathcal{L}^{2}}
+C_{\delta}\|g\|^2_{\mathcal{L}^{2}}.$$  \\
If $(p-1)k>1$, then
$$|\tau(g)|\leq C\|g\|_{\mathcal{L}^{p}}.$$
\end{lemm}

Denote the energy and energy dissipation functionals for $(\rho,u,g)$ as follows:
$$E(t)=\|\rho\|^2_{H^{s}}+\|u\|^2_{H^{s}}+\|g\|^2_{H^{s}(\mathcal{L}^{2})},$$
and
$$D(t)=\|\nabla\rho\|^2_{H^{s-1}}+\mu\|\nabla u\|^2_{H^{s}}+(\mu+\mu')\|divu\|^2_{H^{s}}+\|\nabla_R g\|^2_{H^{s}(\mathcal{L}^{2})}.$$

To study the $L^2$ decay rate, we recall the following global existence of strong solutions for \eqref{eq1}.
\begin{theo}\cite{2021Luo}\label{th1}
Let $d=2~and~s>2$. Let $(\rho,u,g)$ be a strong solution of \eqref{eq1} with the initial data $(\rho_0,u_0,g_0)$ satisfying the conditions $\int_B g_{0}\psi_{\infty}dR=0$ and $1+g_0>0$. Then, there exists some sufficiently small constant $\epsilon_0$ such that if
\begin{align}
E(0)=\|\rho_0\|^2_{H^s}+\|u_0\|^2_{H^s}+\|g_0\|^2_{H^s(\mathcal{L}^2)}\leq \epsilon_0,
\end{align}
then \eqref{eq1} admits a unique global strong solution $(\rho,u,g)$ with $\int_B g\psi_{\infty}dR=0$ and $1+g>0$, and we have
\begin{align}
\sup_{t\in[0,+\infty)} E(t)+\int_{0}^{\infty}D(t)dt\leq C_0 E(0),
\end{align}
where $C_0>1$ is a constant.
\end{theo}

\section{The $L^2$ decay rate}
This section is devoted to investigating the long time behaviour for the compressible FENE dumbbell model with $d=2$. Since the additional stress tensor $\tau$ does not decay fast enough, we failed to use the bootstrap argument as in \cite{Schonbek1985,Luo-Yin2}. To deal with this term, we consider the coupling effect between $\rho$, $u$ and $g$. Motivated by \cite{He2009}, \cite{Luo-Yin} and \cite{2018Global}, we obtain the initial $L^2$ decay rate by taking Fourier transform in \eqref{eq1} and using the Fourier splitting method in following Proposition.
\begin{prop}\label{prop1}
Let $d=2$. Under the condition in Theorem \ref{th2}, for any $l\in N^{+}$, then there exists a constant $C$ such that
\begin{align}\label{decay1}
E_\eta(t)\leq C\ln^{-l}(e+t)
\end{align}
\end{prop}
\begin{proof}
Consider that
$$E_\eta(t)=\sum_{n=0,s}(\|h(\rho)^{\frac 1 2}\Lambda^n\rho\|^2_{L^{2}}+\|(1+\rho)^{\frac 1 2}\Lambda^nu\|^2_{L^{2}})+\lambda\|g\|^2_{H^{s}(\mathcal{L}^{2})}+2\eta\sum_{m=0,s-1}\int_{\mathbb{R}^{2}} \Lambda^m u\nabla\Lambda^m \rho dx,$$
and
$$D_\eta(t)=\eta\gamma\|\nabla\rho\|^2_{H^{s-1}}+\mu\|\nabla u\|^2_{H^{s}}+(\mu+\mu')\|divu\|^2_{H^{s}}+\lambda\|\nabla_R g\|^2_{H^{s}(\mathcal{L}^{2})}.$$
For some sufficiently small constant $\eta>0$, we obtain $E(t)\sim E_\eta(t)$ and $D(t)\sim D_\eta(t)$.

In the proof of Theorem \ref{th1}, see \cite{2021Luo}, we have the following global energy estimation:
\begin{align}\label{ineq1}
\frac d {dt} E_\eta+D_\eta\leq 0.
\end{align}
Define $S_0(t)=\{\xi:|\xi|^2\leq 2C_2\frac {f'(t)} {f(t)}\}$ with $f(t)=\ln^{3}(e+t)$ and $C_2$ large enough. According to Schonbek's strategy, we obtain
\begin{align}\label{ineq2}
&\frac d {dt} [f(t)E_\eta(t)]+C_2 f'(t)(\mu\|u\|^2_{H^{s}}+\eta\gamma\|\rho\|^2_{H^{s-1}})
+f(t)\|\nabla_R g\|^2_{H^{s}(\mathcal{L}^{2})}  \\ \notag
&\leq Cf'(t)\int_{S_0(t)}|\hat{u}(\xi)|^2+|\hat{\rho}(\xi)|^2d\xi+f'(t)\|\Lambda^s \rho\|^2_{L^2}.
\end{align}

We now focus on the $L^2$ estimate to the low frequency part of $\rho$ and $u$. Applying Fourier transform to \eqref{eq1}, we obtain
\begin{align}\label{eq2}
\left\{
\begin{array}{ll}
\hat{\rho}_t+i\xi_{k} \hat{u}^k=\hat{F},  \\[1ex]
\hat{u}^{j}_t+\mu|\xi|^2 \hat{u}^j+(\mu+\mu')\xi_{j} \xi_{k} \hat{u}^k+i\xi_{j} \gamma\hat{\rho}-i\xi_{k} \hat{\tau}^{jk}=\hat{G}^j,  \\[1ex]
\hat{g}_t+\mathcal{L}\hat{g}-i\xi_{k} \hat{u}^j R_j \partial_{R_k}\mathcal{U}+i\xi_{k} \hat{u}^k=\hat{H}, \\[1ex]
\end{array}
\right.
\end{align}
where $F=-div(\rho u)$, $G=-u\cdot\nabla u+[i(\rho)-1](div\Sigma(u)+div \tau)+[\gamma-h(\rho)]\nabla\rho$ and $H=-u\cdot\nabla g-\frac 1 {\psi_\infty} \nabla_R\cdot(\nabla u Rg\psi_\infty)$.  \\
From \eqref{eq2}, using the fact $\int_{B}i\xi_{k} \hat{u}^k\bar{\hat{g}}\psi_\infty dR=0$, we deduce that
\begin{align}\label{eq3}
\left\{
\begin{array}{ll}
\frac 1 2 \frac d {dt} |\hat{\rho}|^2+\mathcal{R}e[i\xi\cdot\hat{u}\bar{\hat{\rho}}]=\mathcal{R}e[\hat{F}\bar{\hat{\rho}}], \\[1ex]
\frac 1 2 \frac d {dt} |\hat{u}|^2+\mathcal{R}e[\gamma\hat{\rho}i\xi\cdot\bar{\hat{u}}]+\mu|\xi|^2 |\hat{u}|^2+(\mu+\mu')|\xi\cdot\hat{u}|^2-\mathcal{R}e[i\xi\otimes\bar{\hat{u}}(t,\xi):\hat{\tau}]=\mathcal{R}e[\hat{G}\cdot\bar{\hat{u}}], \\[1ex]
\frac 1 2 \frac d {dt} \|\hat{g}\|^2_{\mathcal{L}^2}+\|\nabla_R \hat{g}\|^2_{\mathcal{L}^2}-\mathcal{R}e[i\xi\otimes\hat{u}:\bar{\hat{\tau}}]
=\mathcal{R}e[\int_{B}\hat{H}\bar{\hat{g}}\psi_\infty dR].
\end{array}
\right.
\end{align}
One can verify that
$$\mathcal{R}e[i\xi\cdot\hat{u}\bar{\hat{\rho}}]+\mathcal{R}e[\hat{\rho}i\xi\cdot\bar{\hat{u}}]
=\mathcal{R}e[i\xi\otimes\bar{\hat{u}}(t,\xi):\hat{\tau}]+\mathcal{R}e[i\xi\otimes\hat{u}:\bar{\hat{\tau}}]=0,$$
which implies that
\begin{align}\label{eq4}
&\frac 1 2 \frac d {dt} (\gamma|\hat{\rho}|^2+|\hat{u}|^2+\|\hat{g}\|^2_{\mathcal{L}^2})+\mu|\xi|^2 |\hat{u}|^2+(\mu+\mu')|\xi\cdot\hat{u}|^2
+\|\nabla_R \hat{g}\|^2_{\mathcal{L}^2}  \\ \notag
&=\mathcal{R}e[\gamma\hat{F}\bar{\hat{\rho}}]+\mathcal{R}e[\hat{G}\cdot\bar{\hat{u}}]+\mathcal{R}e[\int_{B}\hat{H}\bar{\hat{g}}\psi_\infty dR].
\end{align}
Multiplying $i\xi\cdot\bar{\hat{u}}$ to the first equation of \eqref{eq2}, multiplying $-i\xi_j\bar{\hat{\rho}}$ with $1\leq j\leq 2$ to the second equation of \eqref{eq2} and taking the real part, we can deduce that
\begin{align}\label{eq5}
\mathcal{R}e[\hat{\rho}_t i\xi\cdot\bar{\hat{u}}]-|\xi\cdot\hat{u}|^2=\mathcal{R}e[\hat{F}i\xi\cdot\bar{\hat{u}}],
\end{align}
and
\begin{align}\label{eq6}
\mathcal{R}e[\hat{\rho}i\xi\cdot\bar{\hat{u}}_t]+\gamma|\xi|^2 |\hat{\rho}|^2+(2\mu+\mu')|\xi|^2 \mathcal{R}e[\hat{\rho}i\xi\cdot\bar{\hat{u }}] -\mathcal{R}e[\hat{\rho}\xi\otimes\xi:\bar{\hat{\tau}}]=\mathcal{R}e[\bar{\hat{G}}\cdot i\xi\hat{\rho}].
\end{align}
It follows from \eqref{eq4}$-$\eqref{eq6} that
\begin{align}\label{eq7}
&\frac 1 2 \frac d {dt} (\gamma|\hat{\rho}|^2+|\hat{u}|^2+\|\hat{g}\|^2_{\mathcal{L}^2}+2(\mu+\mu')\mathcal{R}e[\hat{\rho}i\xi\cdot\bar{\hat{u}}])+\mu|\xi|^2 |\hat{u}|^2+(\mu+\mu')\gamma|\xi|^2 |\hat{\rho}|^2
+\|\nabla_R \hat{g}\|^2_{\mathcal{L}^2}  \\ \notag
&=-(\mu+\mu')(2\mu+\mu')|\xi|^2 \mathcal{R}e[\hat{\rho}i\xi\cdot\bar{\hat{u }}] +(\mu+\mu')\mathcal{R}e[\hat{\rho}\xi\otimes\xi:\bar{\hat{\tau}}]+(\mu+\mu')\mathcal{R}e[\hat{F}i\xi\cdot\bar{\hat{u}}]  \\ \notag
&+(\mu+\mu')\mathcal{R}e[\bar{\hat{G}}\cdot i\xi\hat{\rho}]+\mathcal{R}e[\gamma\hat{F}\bar{\hat{\rho}}]+\mathcal{R}e[\hat{G}\cdot\bar{\hat{u}}]+\mathcal{R}e[\int_{B}\hat{H}\bar{\hat{g}}\psi_\infty dR].
\end{align}
Consider $\xi\in S_0(t)$ and sufficiently large $t$, by Lemma \ref{Lemma1} and \ref{Lemma2}, we deduce that
\begin{align}\label{ineq3}
&|\hat{\rho}|^2+|\hat{u}|^2+\|\hat{g}\|^2_{\mathcal{L}^2}
\leq C(|\hat{\rho}_0|^2+|\hat{u}_0|^2+\|\hat{g}_0\|^2_{\mathcal{L}^2})+C\int_{0}^{t}|\hat{G}\cdot\bar{\hat{u}}|+|\widehat{\rho u}|^2+|\hat{G}|^2 ds  \\ \notag
&+C_\delta\int_{0}^{t}\int_{B}\psi_\infty|\mathcal{F}(u\cdot\nabla g)|^2+\psi_\infty|\mathcal{F}(\nabla u\cdot{R}g)|^2 dRds.
\end{align}
Integrating over $S_0(t)$ with $\xi$, then we have the following estimation to \eqref{eq2}:
\begin{align}\label{ineq4}
&\int_{S_0(t)}|\hat{\rho}|^2+|\hat{u}|^2+\|\hat{g}\|^2_{\mathcal{L}^2}d\xi
\leq C\int_{S_0(t)} |\hat{\rho}_0|^2+|\hat{u}_0|^2+\|\hat{g}_0\|^2_{\mathcal{L}^2}d\xi+C\int_{S_0(t)}\int_{0}^{t}|\hat{G}\cdot\bar{\hat{u}}|+|\widehat{\rho u}|^2+|\hat{G}|^2dsd\xi  \\ \notag
&+C_\delta\int_{S_0(t)}\int_{0}^{t}\int_{B}\psi_\infty|\mathcal{F}(u\cdot\nabla g)|^2+\psi_\infty|\mathcal{F}(\nabla u\cdot{R}g)|^2 dRdsd\xi.
\end{align}
If $E(0)<\infty$ and $(\rho_0,u_0,g_0)\in \dot{B}^{-1}_{2,\infty}\times \dot{B}^{-1}_{2,\infty}\times \dot{B}^{-1}_{2,\infty}(\mathcal{L}^2)$, applying Proposition \ref{pro0}, we have
\begin{align*}
\int_{S_0(t)}(|\hat{\rho}_0|^2+|\hat{u}_0|^2+\|\hat{g}_0\|^2_{\mathcal{L}^2})d\xi
&\leq\sum_{j\leq \log_2[\frac {4} {3}C_2^{\frac 1 2 }\sqrt{\frac {f'(t)}{f(t)}}]}\int_{\mathbb{R}^{d}} 2\varphi^2(2^{-j}\xi)(|\hat{\rho}_0|^2+|\hat{u}_0|^2+\|\hat{g}_0\|^2_{\mathcal{L}^2})d\xi \\
&\leq\sum_{j\leq \log_2[\frac {4} {3}C_2^{\frac 1 2 }\sqrt{\frac {f'(t)}{f(t)}}]}(\|\dot{\Delta}_j u_0\|^2_{L^2}+\|\dot{\Delta}_j \rho_0\|^2_{L^2}+\|\dot{\Delta}_j g_0\|^2_{L^2(\mathcal{L}^2)}) \\
&\leq\sum_{j\leq \log_2[\frac {4} {3}C_2^{\frac 1 2 }\sqrt{\frac {f'(t)}{f(t)}}]}2^{2j}(\|u_0\|^2_{\dot{B}^{-1}_{2,\infty}}+\|\rho_0\|^2_{\dot{B}^{-1}_{2,\infty}}+\|g_0\|^2_{\dot{B}^{-1}_{2,\infty}(\mathcal{L}^2)}) \\
&\leq C\frac {f'(t)}{f(t)}(\|u_0\|^2_{\dot{B}^{-1}_{2,\infty}}+\|\rho_0\|^2_{\dot{B}^{-1}_{2,\infty}}+\|g_0\|^2_{\dot{B}^{-1}_{2,\infty}(\mathcal{L}^2)}).
\end{align*}
Thanks to Minkowski's inequality and Theorem \ref{th1}, we get
\begin{align}\label{ineq5}
\int_{S_0(t)}\int_{0}^{t}|\hat{G}|^2dsd\xi
&\leq C\int_{S_0(t)}d\xi \int_{0}^{t}\|\hat{G}^2\|_{L^{\infty}}ds \\ \notag
&\leq C\frac {f'(t)}{f(t)}.
\end{align}
Using Theorem \ref{th1} and Lemma \ref{Lemma1}, we get
\begin{align}\label{ineq6}
&\int_{S_0(t)}\int_{0}^{t}\int_{B}\psi_\infty|\mathcal{F}(u\cdot\nabla g)|^2+\psi_\infty|\mathcal{F}(\nabla u\cdot{R}g)|^2 dRdsd\xi \\ \notag
&\leq C\frac {f'(t)}{f(t)} \int_{0}^{t}\|u\|^2_{L^{2}}\|\nabla g\|^2_{L^{2}(\mathcal{L}^{2})}+\|\nabla u\|^2_{L^{2}}\|g\|^2_{L^{2}(\mathcal{L}^{2})}ds \\ \notag
&\leq C\frac {f'(t)}{f(t)}.
\end{align}
Similarly, we have
\begin{align}\label{ineq7}
&\int_{S_0(t)}\int_{0}^{t}|\hat{G}\cdot\bar{\hat{u}}|+|\widehat{\rho u}|^2  dsd\xi
=\int_{0}^{t}\int_{S_0(t)}|\hat{G}\cdot\bar{\hat{u}}|+|\widehat{\rho u}|^2 d\xi ds  \\ \notag
&\leq C\sqrt{\frac {f'(t)} {f(t)}} \int_{0}^{t}(\|u\|^2_{L^{2}}+\|\rho\|^2_{L^{2}})D(s)^{\frac 1 2}ds+C\frac {f'(t)} {f(t)} \int_{0}^{t}\|u\|^2_{L^{2}}\|\rho\|^2_{L^{2}}ds  \\ \notag
&\leq C\sqrt{\frac {f'(t)} {f(t)}}(1+t)^{\frac 1 2}+C\frac {f'(t)} {f(t)}(1+t).
\end{align}
Plugging the above estimates into \eqref{ineq4}, we obtain
\begin{align}\label{ineq8}
\int_{S_0(t)}|\hat{\rho}|^2+|\hat{u}|^2 d\xi\leq C\ln^{-\frac 1 2}(e+t).
\end{align}
According to \eqref{ineq2} and \eqref{ineq8}, we deduce that
\begin{align}\label{ineq9}
&\frac d {dt} [f(t)E_\eta(t)]+C_2 f'(t)(\mu\|u\|^2_{H^{s}}+\eta\gamma\|\rho\|^2_{H^{s-1}})
+f(t)\|\nabla_R g\|^2_{H^{s}(\mathcal{L}^{2})}  \\ \notag
&\leq Cf'(t)\ln^{-\frac 1 2}(e+t)+2f'(t)\|\Lambda^s \rho\|^2_{L^2},
\end{align}
which implies that
\begin{align*}
f(t)E_\eta(t)\leq C\int_{0}^{t}f'(s)\ln^{-\frac 1 2}(e+s)ds+C\int_{0}^{t}2f'(s)\|\Lambda^s \rho\|^2_{L^2}ds \leq C\ln^{\frac 5 2}(e+t).
\end{align*}
We thus get
\begin{align}\label{ineq10}
E_\eta(t)\leq C\ln^{-\frac 1 2}(e+t).
\end{align}
We improve the $L^2$ decay rate in \eqref{ineq10} by using the bootstrap argument.
According to \eqref{ineq7} and \eqref{ineq10}, we have
\begin{align}\label{ineq11}
&\int_{S_0(t)}\int_{0}^{t}|\hat{G}\cdot\bar{\hat{u}}|+|\widehat{\rho u}|^2  dsd\xi  \\ \notag
&\leq C\sqrt{\frac {f'(t)} {f(t)}} \int_{0}^{t}(\|u\|^2_{L^{2}}+\|\rho\|^2_{L^{2}})D(s)^{\frac 1 2}ds+C\frac {f'(t)} {f(t)} \int_{0}^{t}\|u\|^2_{L^{2}}\|\rho\|^2_{L^{2}}ds  \\ \notag
&\leq C\sqrt{\frac {f'(t)} {f(t)}}(1+t)^{\frac 1 2}\ln^{-\frac 1 2}(e+t)+C\frac {f'(t)} {f(t)}(1+t)\ln^{-1}(e+t),
\end{align}
where in the last inequality we have used the fact that
\begin{align*}
\lim_{t\rightarrow\infty}\frac {\int_{0}^{t}\ln^{-1}(e+s)ds} {(1+t)\ln^{-1}(e+t)}=\lim_{t\rightarrow\infty}\frac {\ln^{-1}(e+t)} {\ln^{-1}(1+t)-\ln^{-2}(e+t)}  =1.
\end{align*}
Then the proof of \eqref{ineq8} implies that
\begin{align}\label{ineq12}
\int_{S_0(t)}|\hat{\rho}|^2+|\hat{u}|^2 d\xi\leq C\ln^{-1}(e+t).
\end{align}
According to \eqref{ineq2} and \eqref{ineq12}, we deduce that
\begin{align*}
E_\eta\leq C\ln^{-1}(e+t).
\end{align*}
Using the bootstrap argument, for any $l\in N^{+}$, we obtain
\begin{align}\label{ineq13}
E_\eta\leq C\ln^{-l}(e+t).
\end{align}
We thus complete the proof of Proposition \ref{prop1}
\end{proof}

In order to improve the decay rate, we estimate the following high order energy:
\begin{align*}
E^1_{\eta}(t)&=\sum_{n=1,s}(\|h(\rho)^{\frac 1 2}\Lambda^n\rho \|^2_{L^2}+\|(1+\rho)^{\frac 1 2}\Lambda^n u\|^2_{L^2})  \\ \notag
&+\|\Lambda^1 g\|^2_{H^{s-1}(\mathcal{L}^{2})}+2\eta\sum_{m=1,s-1}\int_{\mathbb{R}^{2}} \Lambda^m u\nabla\Lambda^m \rho dx,
\end{align*}
and
$$D^1_{\eta}(t)=\eta\gamma\|\nabla\Lambda^1 \rho\|^2_{H^{s-2}}+\mu\|\nabla\Lambda^1 u\|^2_{H^{s-1}}+(\mu+\mu')\|div\Lambda^1u\|^2_{H^{s-1}}+\|\Lambda^1\nabla_R g\|^2_{H^{s-1}(\mathcal{L}^{2})}.$$
The following Lemma can be proved by the standard energy method. Thus we omit the
proof here.
\begin{lemm}\label{Lemma3}
Let $(\rho_,u,g)\in L^{\infty}([0,+\infty);H^s\times H^s\times H^s(\mathcal{L}^2))$  be global strong solutions constructed in Theorem \ref{th1}. If $t\in(0,+\infty)$, then we have
\begin{align}
\frac d {dt}E^1_{\eta}(t)+D^1_{\eta}(t)\leq 0.
\end{align}
\end{lemm}

The following proposition is about the high order energy estimate.
\begin{prop}\label{prop2}
Let $d=2$. Under the condition in Theorem \ref{th2}, for any $l\in N^{+}$, then there exists a constant $C$ such that
\begin{align}\label{decay2}
E^1_{\eta}(t)\leq C(1+t)^{-1}\ln^{-l}(e+t),
\end{align}
and
\begin{align}\label{decay3}
\int_{0}^{t}(1+s)f(s)\|\Lambda^1 \nabla_R g\|^2_{H^{s-1}(\mathcal{L}^{2})}ds\leq C.
\end{align}
\end{prop}
\begin{proof}
Applying Lemma \ref{Lemma3}, we have
\begin{align}\label{ineq14}
\frac d {dt} E^1_{\eta}+D^1_{\eta}\leq 0,
\end{align}
which implies that
\begin{align}\label{ineq15}
&\frac d {dt} [f(t)E^1_{\eta}]+C_2 f'(t)(\mu\|\Lambda^1 u\|^2_{H^{s-1}}+\eta\gamma\|\Lambda^1 \rho\|^2_{H^{s-2}})
+f(t)\|\Lambda^1 \nabla_R g\|^2_{H^{s-1}(\mathcal{L}^{2})}  \\ \notag
&\leq Cf'(t)\int_{S_0(t)}|\xi|^2(|\hat{u}(\xi)|^2+|\hat{\rho}(\xi)|^2) d\xi+f'(t)\|\Lambda^s \rho\|^2_{L^{2}}.
\end{align}
According to \eqref{ineq13}, we have
\begin{align*}
f'(t)\int_{S_0(t)}|\xi|^2(|\hat{u}(\xi)|^2+|\hat{\rho}(\xi)|^2) d\xi\leq C (1+t)^{-2}\ln^{-l+1}(e+t).
\end{align*}
This together with \eqref{ineq13}, \eqref{ineq15} and \eqref{ineq1} ensures that
\begin{align*}
(1+t)\ln^{l+1}(e+t)E^1_{\eta}&\leq C+C\ln(e+t)+C\int_{0}^{t}\ln^{l}(e+s)\|\Lambda^s \rho\|^2_{L^{2}}ds+C\int_{0}^{t}\ln^{l+1}(e+s)E^1_{\eta}ds \\
&\leq C\ln(e+t)+C\int_{0}^{t}\ln^{l+1}(e+s)D_\eta ds  \\
&\leq C\ln(e+t)+C\int_{0}^{t}(1+s)^{-1}\ln^{l}(e+s)E_\eta ds \\
&\leq C\ln(e+t),
\end{align*}
which implies that
\begin{align}\label{ineq16}
E^1_{\eta}\leq C (1+t)^{-1}\ln^{-l}(e+t).
\end{align}
Using \eqref{ineq15} and \eqref{ineq16} with $l\geq 5$, we have the following estimate for $\nabla_R g$:
\begin{align}\label{ineq17}
&\int_{0}^{t}(1+s)f(s)\|\Lambda^1 \nabla_R g\|^2_{H^{s-1}(\mathcal{L}^{2})}ds  \\ \notag
&\leq C+\int_{0}^{t}(1+s)f'(s)\|\Lambda^s \rho\|^2_{L^{2}}ds+C\int_{0}^{t}f(s)E^1_{\eta}ds\leq C.
\end{align}
We thus complete the proof of Proposition \ref{prop2}.
\end{proof}

By virtue of the standard method, one can not obtain the optimal decay rate. However, we can obtain a weak result as follow.
\begin{prop}\label{prop3}
Under the condition in Theorem \ref{th2}, then there exists a constant $C$ such that
\begin{align}\label{decay4}
E_{\eta}(t)\leq C(1+t)^{-\frac 1 2},
\end{align}
and
\begin{align}\label{decay5}
E^1_{\eta}(t)\leq C(1+t)^{-\frac 3 2}.
\end{align}
\end{prop}
\begin{proof}
Define $S(t)=\{\xi:|\xi|^2\leq C_2(1+t)^{-1}\}$ where the constant $C_2$ will be chosen later on. Using Schonbek's strategy, we split the phase space into two time-dependent domain:
$$\|\nabla u\|^2_{H^s}=\int_{S(t)}(1+|\xi|^{2s})|\xi|^2|\hat{u}(\xi)|^2 d\xi+\int_{S(t)^c}(1+|\xi|^{2s})|\xi|^2|\hat{u}(\xi)|^2 d\xi.$$
Then we can easily deduce that
$$\frac {C_2} {1+t} \int_{S(t)^c}(1+|\xi|^{2s})|\hat{u}(\xi)|^2 d\xi\leq\|\nabla u\|^2_{H^s},$$
and
$$\frac {C_2} {1+t} \int_{S(t)^c}(1+|\xi|^{2s-2})|\hat{\rho}(\xi)|^2 d\xi\leq\|\nabla \rho\|^2_{H^{s-1}}.$$
According to \eqref{ineq1}, we have
\begin{align}\label{ineq18}
\frac d {dt} E_\eta(t)+\frac {\mu C_2} {1+t}\| u\|^2_{H^s}+\frac {\eta\gamma C_2} {1+t}\|\rho\|^2_{H^{s-1}}
+\|\nabla_R g\|^2_{H^{s}(\mathcal{L}^{2})}\leq \frac {CC_2} {1+t}\int_{S(t)}|\hat{u}(\xi)|^2+|\hat{\rho}(\xi)|^2 d\xi.
\end{align}
Integrating \eqref{ineq3} over $S(t)$ with $\xi$, then we have
\begin{align}\label{ineq19}
&\int_{S(t)}|\hat{\rho}|^2+|\hat{u}|^2+\|\hat{g}\|^2_{\mathcal{L}^2}d\xi
\leq C\int_{S(t)} (|\hat{\rho}_0|^2+|\hat{u}_0|^2+\|\hat{g}_0\|^2_{\mathcal{L}^2})d\xi+C\int_{S(t)}\int_{0}^{t}|\hat{G}\cdot\bar{\hat{u}}|+|\widehat{\rho u}|^2+|\hat{G}|^2dsd\xi  \\ \notag
&+C_\delta\int_{S(t)}\int_{0}^{t}\int_{B}\psi_\infty|\mathcal{F}(u\cdot\nabla g)|^2+\psi_\infty|\mathcal{F}(\nabla u\cdot{R}g)|^2 dRdsd\xi.
\end{align}
If $E(0)<\infty$ and $(\rho_0,u_0,g_0)\in \dot{B}^{-1}_{2,\infty}\times \dot{B}^{-1}_{2,\infty}\times \dot{B}^{-1}_{2,\infty}(\mathcal{L}^2)$, applying Propositon \ref{pro0}, we have
\begin{align*}
\int_{S(t)}(|\hat{\rho}_0|^2+|\hat{u}_0|^2+\|\hat{g}_0\|^2_{\mathcal{L}^2})d\xi
&\leq\sum_{j\leq \log_2[\frac {4} {3}C_2^{\frac 1 2 }(1+t)^{-\frac 1 2}]}\int_{\mathbb{R}^{d}} 2\varphi^2(2^{-j}\xi)(|\hat{\rho}_0|^2+|\hat{u}_0|^2+\|\hat{g}_0\|^2_{\mathcal{L}^2})d\xi \\
&\leq\sum_{j\leq \log_2[\frac {4} {3}C_2^{\frac 1 2 }(1+t)^{-\frac 1 2}]}(\|\dot{\Delta}_j u_0\|^2_{L^2}+\|\dot{\Delta}_j \rho_0\|^2_{L^2}+\|\dot{\Delta}_j g_0\|^2_{L^2(\mathcal{L}^2)}) \\
&\leq\sum_{j\leq \log_2[\frac {4} {3}C_2^{\frac 1 2 }(1+t)^{-\frac 1 2}]}2^{2j}(\|u_0\|^2_{\dot{B}^{-1}_{2,\infty}}+\|\rho_0\|^2_{\dot{B}^{-1}_{2,\infty}}+\|g_0\|^2_{\dot{B}^{-1}_{2,\infty}(\mathcal{L}^2)}) \\
&\leq C(1+t)^{-1}(\|u_0\|^2_{\dot{B}^{-1}_{2,\infty}}+\|\rho_0\|^2_{\dot{B}^{-1}_{2,\infty}}+\|g_0\|^2_{\dot{B}^{-1}_{2,\infty}(\mathcal{L}^2)}).
\end{align*}
Thanks to Minkowski's inequality, we get
\begin{align*}
\int_{S(t)}\int_{0}^{t}|\widehat{\rho u}|^2dsd\xi
&=\int_{0}^{t}\int_{S(t)}|\widehat{\rho u}|^2 d\xi ds  \\ \notag
&\leq C\int_{S(t)}d\xi \int_{0}^{t}\||\widehat{\rho u}|^2\|_{L^{\infty}}ds \\ \notag
&\leq C(1+t)^{-1} \int_{0}^{t}\|\rho\|^2_{L^{2}}\|u\|^2_{L^{2}}ds,
\end{align*}
and
\begin{align*}
\int_{S(t)}\int_{0}^{t}|\hat{G}|^2dsd\xi
&\leq C\int_{S(t)}d\xi \int_{0}^{t}\||\hat{G}|^2\|_{L^{\infty}}ds \\ \notag
&\leq C(1+t)^{-1}.
\end{align*}
Thanks to Lemma \ref{Lemma1} and Lemma \ref{Lemma2}, we have
\begin{align}\label{ineq20}
&\int_{S(t)}\int_{0}^{t}|\hat{G}\cdot\bar{\hat{u}}|dsd\xi
\leq C(\int_{S(t)}d\xi)^{\frac 1 2} \int_{0}^{t}\|\hat{G}\cdot\bar{\hat{u}}\|_{L^{2}}ds \\ \notag
&\leq C(1+t)^{-\frac 1 2} \int_{0}^{t}(\|u\|^2_{L^{2}}+\|\rho\|^2_{L^{2}})(\|\nabla u\|_{H^{1}}+\|\nabla\rho\|_{L^{2}}+\|\nabla\nabla_R g\|_{L^{2}(\mathcal{L}^{2})})ds.
\end{align}
Using Theorem \ref{th1} and Lemma \ref{Lemma1}, we get
\begin{align}\label{ineq21}
&\int_{S(t)}\int_{0}^{t}\int_{B}\psi_\infty|\mathcal{F}(u\cdot\nabla g)|^2+\psi_\infty|\mathcal{F}(\nabla u\cdot{R}g)|^2 dRdsd\xi \\ \notag
&\leq C(1+t)^{-1} \int_{0}^{t}\|u\|^2_{L^{2}}\|\nabla g\|^2_{L^{2}(\mathcal{L}^{2})}+\|\nabla u\|^2_{L^{2}}\|g\|^2_{L^{2}(\mathcal{L}^{2})}ds \\ \notag
&\leq C(1+t)^{-1}.
\end{align}
Plugging the above estimates into \eqref{ineq19}, we obtain
\begin{align}\label{ineq22}
&\int_{S(t)}|\hat{\rho}(t,\xi)|^2+|\hat{u}(t,\xi)|^2 d\xi\leq  C(1+t)^{-1}+C(1+t)^{-1} \int_{0}^{t}\|\rho\|^2_{L^{2}}\|u\|^2_{L^{2}}ds  \\ \notag
&+C(1+t)^{-\frac 1 2} \int_{0}^{t}(\|u\|^2_{L^{2}}+\|\rho\|^2_{L^{2}})(\|\nabla u\|_{H^{1}}+\|\nabla\rho\|_{L^{2}}+\|\nabla\nabla_R g\|_{L^{2}(\mathcal{L}^{2})})ds.
\end{align}
According to \eqref{ineq18} and \eqref{ineq22}, we deduce that
\begin{align*}
&\frac d {dt} E_\eta(t)+\frac {\mu C_2} {1+t}\| u\|^2_{H^s}+\frac {\eta\gamma C_2} {1+t}\|\rho\|^2_{H^{s-1}}
+\|\nabla_R g\|^2_{H^{s}(\mathcal{L}^{2})}  \\ \notag
&\leq \frac {CC_2} {1+t}[(1+t)^{-1}+(1+t)^{-1} \int_{0}^{t}\|\rho\|^2_{L^{2}}\|u\|^2_{L^{2}}ds  \\ \notag
&+(1+t)^{-\frac 1 2} \int_{0}^{t}(\|u\|^2_{L^{2}}+\|\rho\|^2_{L^{2}})(\|\nabla u\|_{H^{1}}+\|\nabla\rho\|_{L^{2}}+\|\nabla\nabla_R g\|_{L^{2}(\mathcal{L}^{2})})ds.]
\end{align*}
If $C_2$ large enough, then we have
\begin{align}\label{ineq23}
(1+t)^{\frac 3 2}E_\eta(t)&\leq C\int_{0}^{t}\|\Lambda^s \rho\|^2_{L^{2}}(1+s)^{\frac 1 2}ds+C(1+t)^{\frac 1 2}+C\int_{0}^{t}(1+s)^{-\frac 1 2} \int_{0}^{s}\|\rho\|^2_{L^{2}}\|u\|^2_{L^{2}}ds'ds  \\ \notag
&+C(1+t)\int_{0}^{t}(\|u\|^2_{L^{2}}+\|\rho\|^2_{L^{2}})(\|\nabla u\|_{H^{1}}+\|\nabla\rho\|_{L^{2}}+\|\nabla\nabla_R g\|_{L^{2}(\mathcal{L}^{2})})ds\\ \notag
&\leq C(1+t)^{\frac 1 2}+C\int_{0}^{t}(1+s)^{-\frac 1 2} \int_{0}^{t}\|\rho\|^2_{L^{2}}\|u\|^2_{L^{2}}ds'ds  \\ \notag
&+C(1+t)\int_{0}^{t}(\|u\|^2_{L^{2}}+\|\rho\|^2_{L^{2}})(\|\nabla u\|_{H^{1}}+\|\nabla\rho\|_{L^{2}}+\|\nabla\nabla_R g\|_{L^{2}(\mathcal{L}^{2})})ds  \\  \notag
&\leq C(1+t)^{\frac 1 2}+C(1+t)^{\frac 1 2}\int_{0}^{t}\|\rho\|^2_{L^{2}}\|u\|^2_{L^{2}}ds  \\ \notag
&+C(1+t)\int_{0}^{t}(\|u\|^2_{L^{2}}+\|\rho\|^2_{L^{2}})(\|\nabla u\|_{H^{1}}+\|\nabla\rho\|_{L^{2}}+\|\nabla\nabla_R g\|_{L^{2}(\mathcal{L}^{2})})ds.
\end{align}
Define $N(t)=\sup_{0\leq s\leq t}(1+s)^{\frac 1 2}E_\eta(s)$. According to \eqref{ineq23}, we get
\begin{align}\label{ineq24}
N(t)&\leq C+C\sup_{0\leq s\leq t}(1+s)^{-\frac 1 2}\int_{0}^{s}\|\rho\|^2_{L^{2}}\|u\|^2_{L^{2}}ds'  \\ \notag
&+C\int_{0}^{t}(\|u\|^2_{L^{2}}+\|\rho\|^2_{L^{2}})(\|\nabla u\|_{H^{1}}+\|\nabla\rho\|_{L^{2}}+\|\nabla\nabla_R g\|_{L^{2}(\mathcal{L}^{2})})ds  \\ \notag
&\leq C+C\sup_{0\leq s\leq t}\int_{0}^{s}\|\rho\|^2_{L^{2}}\|u\|^2_{L^{2}}(1+s')^{-\frac 1 2}ds'  \\ \notag
&+C\int_{0}^{t}N(s)(1+s)^{-\frac 1 2}(\|\nabla u\|_{H^{1}}+\|\nabla\rho\|_{L^{2}}+\|\nabla\nabla_R g\|_{L^{2}(\mathcal{L}^{2})})ds \\ \notag
&\leq C+C\int_{0}^{t}\|\rho\|^2_{L^{2}}N(s)(1+s)^{-1}ds  \\ \notag
&+C\int_{0}^{t}N(s)(1+s)^{-\frac 1 2}(\|\nabla u\|_{H^{1}}+\|\nabla\rho\|_{L^{2}}+\|\nabla\nabla_R g\|_{L^{2}(\mathcal{L}^{2})})ds.
\end{align}
Applying Gronwall's inequality, Propositions \ref{prop1} and \ref{prop2}, we obtain $N(t)\leq C$,
which implies that
\begin{align}\label{ineq25}
E_\eta(t)\leq C(1+t)^{-\frac 1 2}.
\end{align}
Applying Lemma \ref{Lemma3}, we have
\begin{align}\label{ineq26}
\frac d {dt} E^1_\eta+D^1_\eta\leq 0,
\end{align}
which implies that
\begin{align}\label{ineq27}
&\frac d {dt} E^1_\eta+\frac { C_2} {1+t}(\mu\|\Lambda^1 u\|^2_{H^{s-1}}+\eta\gamma\|\Lambda^1 \rho\|^2_{H^{s-2}})
+\|\Lambda^1 \nabla_R g\|^2_{H^{s-1}(\mathcal{L}^{2})} \\ \notag
&\leq \frac {CC_2} {1+t}\int_{S(t)}|\xi|^2(|\hat{u}(\xi)|^2+|\hat{\rho}(\xi)|^2) d\xi.
\end{align}
According to \eqref{ineq25}, we have
\begin{align*}
\frac {CC_2} {1+t}\int_{S(t)}|\xi|^2(|\hat{u}(\xi)|^2+|\hat{\rho}(\xi)|^2) d\xi\leq C{C_2}^2 (1+t)^{-2}(\|\rho\|^2_{L^2}+\|u\|^2_{L^2})\leq C (1+t)^{-\frac 5 2}.
\end{align*}
This together with \eqref{ineq1}, \eqref{ineq25} and \eqref{ineq27} ensure that
\begin{align*}
(1+t)^{\frac 5 2}E^1_\eta&\leq C(1+t)+C\int_{0}^{t}\|\Lambda^s \rho\|^2_{L^{2}}(1+s)^{\frac 3 2}ds  \\ \notag
&\leq C(1+t)+C\int_{0}^{t}E_\eta(s)(1+s)^{\frac 1 2}ds   \\ \notag
&\leq C(1+t),
\end{align*}
which implies that
\begin{align}\label{ineq28}
E^1_\eta\leq C(1+t)^{-\frac 3 2}.
\end{align}
Therefore, we complete the proof of Proposition \ref{prop3}.
\end{proof}

\begin{rema}
The proposition \ref{prop3} indicates that
$$\|\rho\|_{L^2}+\|u\|_{L^2}\leq C(1+t)^{-\frac 1 4}.$$
Combining with the incompressible FENE model and CNS system, one can see that this is not the optimal time decay.
\end{rema}

In order to obtain optimal $L^2$ decay rate, we have to use the estimate of lower frequency. The following lemma show that the decay rate  implies that the solution will belong to some Besov space with negative index.
\begin{lemm}\label{prop4}
Let $0<\alpha,\sigma\leq 1$ and $\sigma<2\alpha$. Assume that $(\rho_0,u_0,g_0)$ satisfy the condition in Theorem \ref{th2}. For any $t\in [0,+\infty)$, if
\begin{align}\label{decay6}
E_{\eta}(t)\leq C(1+t)^{-\alpha},~~~~E^1_{\eta}(t)\leq C(1+t)^{-\alpha-1}~~~~and~~~~\int_{0}^{t}(1+s)^{\alpha}\|\Lambda^1 \nabla_R g\|^2_{L^{2}(\mathcal{L}^{2})}ds\leq C,
\end{align}
then we have
\begin{align}\label{ineq29}
(\rho,u,g)\in L^{\infty}(0,\infty;\dot{B}^{-\sigma}_{2,\infty})\times L^{\infty}(0,\infty;\dot{B}^{-\sigma}_{2,\infty})\times L^{\infty}(0,\infty;\dot{B}^{-\sigma}_{2,\infty}(\mathcal{L}^2)).
\end{align}
\begin{proof}
Applying $\dot{\Delta}_j$ to \eqref{eq1}, we get
\begin{align}\label{eq8}
\left\{
\begin{array}{ll}
\dot{\Delta}_j\rho_t+div~\dot{\Delta}_j u=\dot{\Delta}_j F,  \\[1ex]
\dot{\Delta}_j u_t-div\Sigma(\dot{\Delta}_j u)+\gamma\nabla\dot{\Delta}_j \rho-div\dot{\Delta}_j\tau=\dot{\Delta}_j G,  \\[1ex]
\dot{\Delta}_j g_t+\mathcal{L}\dot{\Delta}_j g-\nabla\dot{\Delta}_j u R_j \partial_{R_k}\mathcal{U}+div \dot{\Delta}_j u=\dot{\Delta}_j H, \\[1ex]
\end{array}
\right.
\end{align}
where $F=-div(\rho u)$, $G=-u\cdot\nabla u+[i(\rho)-1](div\Sigma(u)+div \tau)+[\gamma-h(\rho)]\nabla\rho$ and $H=-u\cdot\nabla g-\frac 1 {\psi_\infty} \nabla_R\cdot(\nabla u Rg\psi_\infty)$.

Using $\int_{B} \dot{\Delta}_j g\psi_\infty dR=0$ and integrating by parts, then we have
\begin{align}\label{ineq30}
&\frac 1 2 \frac d {dt}(\gamma\|\dot{\Delta}_j \rho\|^2_{L^2}+\|\dot{\Delta}_j u\|^2_{L^2}+\|\dot{\Delta}_j g\|^2_{L^2(\mathcal{L}^{2})}) \\ \notag
&+\mu\|\nabla\dot{\Delta}_j u\|^2_{L^2}+(\mu+\mu')\|div \dot{\Delta}_j u\|^2_{L^2}+\|\nabla_R \dot{\Delta}_j g\|^2_{L^2(\mathcal{L}^{2})}   \\ \notag
&=\int_{\mathbb{R}^{2}} \gamma\dot{\Delta}_j F\dot{\Delta}_j \rho dx+\int_{\mathbb{R}^{2}} \dot{\Delta}_j G\dot{\Delta}_j u dx+\int_{\mathbb{R}^{2}}\int_{B} \dot{\Delta}_j H\dot{\Delta}_j g \psi_\infty dxdR   \\ \notag
&\leq C(\|\dot{\Delta}_j F\|_{L^2}\|\dot{\Delta}_j \rho\|_{L^2}+\|\dot{\Delta}_j G\|_{L^2}\|\dot{\Delta}_j u\|_{L^2}) \\ \notag
&+C(\|\dot{\Delta}_j (u\nabla g)\|^2_{L^2(\mathcal{L}^{2})}+\|\dot{\Delta}_j (\nabla uRg)\|^2_{L^2(\mathcal{L}^{2})})+\frac 1 2 \|\nabla_R \dot{\Delta}_j g\|^2_{L^2(\mathcal{L}^{2})} .
\end{align}
Multiplying both sides of \eqref{ineq30} by $2^{-2j\sigma}$ and taking $l^\infty$ norm, we get
\begin{align}\label{ineq31}
&\frac d {dt}(\gamma\|\rho\|^2_{\dot{B}^{-\sigma}_{2,\infty}}+\|u\|^2_{\dot{B}^{-\sigma}_{2,\infty}}+\|g\|^2_{\dot{B}^{-\sigma}_{2,\infty}(\mathcal{L}^{2})})  \\ \notag
&\leq C(\|F\|_{\dot{B}^{-\sigma}_{2,\infty}}\|\rho\|_{\dot{B}^{-\sigma}_{2,\infty}}+\|G\|_{\dot{B}^{-\sigma}_{2,\infty}}\|u\|_{\dot{B}^{-\sigma}_{2,\infty}}+ \|\nabla uRg\|^2_{\dot{B}^{-\sigma}_{2,\infty}(\mathcal{L}^{2})}+\|u\nabla g\|^2_{\dot{B}^{-\sigma}_{2,\infty}(\mathcal{L}^{2})}).
\end{align}
Define $M(t)=\sum_{0\leq s\leq t} \|\rho\|_{\dot{B}^{-\sigma}_{2,\infty}}+\|u\|_{\dot{B}^{-\sigma}_{2,\infty}}+\|g\|_{\dot{B}^{-\sigma}_{2,\infty}(\mathcal{L}^{2})}$. According to \eqref{ineq31}, we deduce that
\begin{align}\label{ineq32}
M^2(t)&\leq M^2(0)+CM(t)\int_0^{t}\|F\|_{\dot{B}^{-\sigma}_{2,\infty}}+\|G\|_{\dot{B}^{-\sigma}_{2,\infty}}ds  \\  \notag
&+C\int_0^{t}\|\nabla uRg\|^2_{\dot{B}^{-\sigma}_{2,\infty}(\mathcal{L}^{2})}+\|u\nabla g\|^2_{\dot{B}^{-\sigma}_{2,\infty}(\mathcal{L}^{2})}ds.
\end{align}
Using Lemmas \ref{Lemma}, \ref{Lemma0} and \eqref{decay6}, we obtain
\begin{align*}
\int_0^{t}\|\nabla uRg\|^2_{\dot{B}^{-\sigma}_{2,\infty}(\mathcal{L}^{2})}+\|u\nabla g\|^2_{\dot{B}^{-\sigma}_{2,\infty}(\mathcal{L}^{2})}ds
&\leq C\int_0^{t}\|\nabla uRg\|^2_{L^{\frac {2} {\sigma+1}}(\mathcal{L}^{2})}+\|u\nabla g\|^2_{L^{\frac {2} {\sigma+1}}(\mathcal{L}^{2})}ds  \\
&\leq C\int_0^{t}\|\nabla u\|^2_{L^2}\|g\|^2_{L^{\frac {2} {\sigma}}(\mathcal{L}^{2})}+\|\nabla g\|^2_{L^2(\mathcal{L}^{2})}\|u\|^2_{L^{\frac {2} {\sigma}}}ds\leq C,
\end{align*}
and
\begin{align*}
\int_0^{t}\|F\|_{\dot{B}^{-\sigma}_{2,\infty}}ds
&\leq C\int_0^{t}\|F\|_{L^{\frac {2} {\sigma+1}}}ds  \\
&\leq C\int_0^{t}\|u\|_{L^\frac {2} {\sigma}}\|\nabla\rho\|_{L^2}+\|div~u\|_{L^2}\|\rho\|_{L^\frac {2} {\sigma}}ds   \\
&\leq C\int_0^{t}\|u\|^{\sigma}_{L^2}\|\nabla u\|^{1-\sigma}_{L^2}\|\nabla\rho\|_{L^2}+\|div~u\|_{L^2}\|\rho\|^{\sigma}_{L^2}\|\nabla\rho\|^{1-\sigma}_{L^2}ds   \\
&\leq C\int_0^{t}(1+s)^{-(1+\alpha-\frac \sigma 2)}ds\leq C.
\end{align*}
By virtue of \eqref{decay6}, we similarly get
\begin{align}\label{ineq33}
\int_0^{t}\|G\|_{\dot{B}^{-\sigma}_{2,\infty}}ds
&\leq C\int_0^{t}(1+s)^{-(1+\alpha-\frac \sigma 2)}ds+C\int_0^{t}\|div~\tau\|_{L^2}\|\rho\|_{L^\frac {2} {\sigma}}ds  \\ \notag
&\leq C+C\int_0^{t}\|\rho\|^{\sigma}_{L^2}\|\nabla\rho\|^{1-\sigma}_{L^2}\|\nabla\nabla_R g\|_{L^2(\mathcal{L}^{2})}ds  \\ \notag
&\leq C+C(\int_0^{t}\|\rho\||^{2\sigma}_{L^2}\|\nabla\rho\|^{2-2\sigma}_{L^2}(1+s)^{-\alpha}ds)^{\frac 1 2}(\int_0^{t}(1+s)^{\alpha}\|\nabla\nabla_R g\|^2_{L^2(\mathcal{L}^{2})}ds)^{\frac 1 2}   \\ \notag
&\leq C+C(\int_0^{t}(1+s)^{-(1+2\alpha-\sigma)}ds)^{\frac 1 2}\leq C.
\end{align}
From \eqref{ineq32}, we have $M^2(t)\leq CM^2(0)+M(t)C+C$. By virtue of interpolation theory, we can deduce that $(\rho_0,u_0,g_0)\in \dot{B}^{-\sigma}_{2,\infty}\times \dot{B}^{-\sigma}_{2,\infty}\times \dot{B}^{-\sigma}_{2,\infty}(\mathcal{L}^2)$ with $0<\sigma\leq1$. Then $M^2(0)\leq C$ implies that $M(t)\leq C$.
\end{proof}

\end{lemm}

Using the above lemma one can prove that the solution belongs to some Besov space. The following proposition indicates that if the solution belongs to some Besov space with negative index, then the decay rate can be improved.

\begin{prop}\label{prop5}
Let $0<\beta,\sigma\leq 1$ and $\frac {1} {2} \leq\alpha$. Assume that $(\rho_0,u_0,g_0)$ satisfy the condition in Theorem \ref{th2}.
For any $t\in [0,+\infty)$, if
\begin{align}\label{decay7}
E_{\eta}(t)\leq C(1+t)^{-\alpha},~~~~E^1_{\eta}(t)\leq C(1+t)^{-\alpha-1}~~~~and~~~~\int_{0}^{t}(1+s)^{\frac 3 4}\|\Lambda^1 \nabla_R g\|^2_{L^{2}(\mathcal{L}^{2})}ds\leq C,
\end{align}
and
\begin{align}\label{ineq34}
(\rho,u,g)\in L^{\infty}(0,\infty;\dot{B}^{-\sigma}_{2,\infty})\times L^{\infty}(0,\infty;\dot{B}^{-\sigma}_{2,\infty})\times L^{\infty}(0,\infty;\dot{B}^{-\sigma}_{2,\infty}(\mathcal{L}^2)),
\end{align}
then there exists a constant $C$ such that
\begin{align}\label{decay8}
E_{\eta}(t)\leq C(1+t)^{-\beta}~~~~and~~~~E^1_{\eta}(t)\leq C(1+t)^{-\beta-1}
\end{align}
where $\beta<\frac {\sigma+1} {2}$ for $\alpha=\frac {1} {2}$ and $\beta=\frac {\sigma+1} {2}$ for $\alpha>\frac {1} {2}$.
\end{prop}
\begin{proof}
According to \eqref{decay7}, we obtain
\begin{align*}
\int_{S(t)}\int_{0}^{t}|\widehat{\rho u}|^2dsd\xi\leq C(1+t)^{-1} \int_{0}^{t}\|\rho\|^2_{L^{2}}\|u\|^2_{L^{2}}ds\leq C(1+t)^{-1} \int_{0}^{t}(1+s)^{-2\alpha}ds\leq C(1+t)^{-\beta}.
\end{align*}
By virtue of \eqref{decay7} and \eqref{ineq34}, we have
\begin{align*}
\int_{S(t)}\int_{0}^{t}|\hat{G}\cdot\bar{\hat{u}}|dsd\xi
&\leq C\int_{0}^{t}\|G\|_{L^{1}}\int_{S(t)}|\hat{u}|d\xi ds \\
&\leq C(1+t)^{-\frac 1 2}\int_{0}^{t}\|G\|_{L^{1}}(\int_{S(t)}|\hat{u}|^2d\xi)^{\frac 1 2} ds  \\
&\leq C(1+t)^{-\frac {\sigma+1} {2}}M(t)\int_{0}^{t}\|G\|_{L^{1}}ds  \\
&\leq C(1+t)^{-\frac {\sigma+1} {2}}\int_{0}^{t}(1+s)^{-\alpha-\frac 1 2}+\|\rho\|_{L^{2}}\|\nabla\nabla_R g\|_{L^2(\mathcal{L}^{2})}ds  \\
&\leq C(1+t)^{-\beta}+C(1+t)^{-\frac {\sigma+1} {2}}(\int_{0}^{t}\|\rho\|^2_{L^{2}}(1+s)^{-\frac 3 4}ds)^{\frac 1 2}  \\
&\leq C(1+t)^{-\beta}.
\end{align*}
 Similar to the proof of Proposition \ref{prop3}, we obtain
\begin{align}\label{ineq35}
\frac d {dt} E_\eta(t)+\frac {\mu C_2} {1+t}\| u\|^2_{H^s}+\frac {\eta\gamma C_2} {1+t}\|\rho\|^2_{H^{s-1}}
+\|\nabla_R g\|^2_{H^{s}(\mathcal{L}^{2})}\leq \frac {CC_2} {1+t}(1+t)^{-\beta}.
\end{align}
Then the proof of \eqref{ineq28} implies that
\begin{align}\label{ineq36}
E_\eta(t)\leq C(1+t)^{-\beta}.
\end{align}
Using Lemma \ref{Lemma3}, we have
\begin{align}\label{ineq37}
&\frac d {dt} E^1_\eta+\frac { C_1} {1+t}(\mu\|\Lambda^1 u\|^2_{H^{s-1}}+\eta\gamma\|\Lambda^1 \rho\|^2_{H^{s-2}})
+\|\Lambda^1 \nabla_R g\|^2_{H^{s-1}(\mathcal{L}^{2})}  \\ \notag
&\leq \frac {CC_1} {1+t}\int_{S(t)}|\xi|^2(|\hat{u}(\xi)|^2+|\hat{\rho}(\xi)|^2) d\xi.
\end{align}
According to \eqref{ineq36}, we have
\begin{align}
\frac {CC_2} {1+t}\int_{S(t)}|\xi|^2(|\hat{u}(\xi)|^2+|\hat{\rho}(\xi)|^2) d\xi\leq C{C_2}^2 (1+t)^{-2}(\|\rho\|^2_{L^2}+\|u\|^2_{L^2})\leq C (1+t)^{-2-\beta}.
\end{align}
Then the proof of \eqref{ineq28} implies that $E^1_\eta\leq C(1+t)^{-1-\beta}$.
We thus finish the proof of Proposition \ref{prop5}.
\end{proof}

{\bf The proof of Theorem \ref{th2}:}  \\
We now improve the decay rate in Proposition \ref{prop3}. According to Propositions \ref{prop2}, \ref{prop3} and Lemma \ref{prop4} with $\sigma=\alpha=\frac 1 2$, we obtain
\begin{align*}
(\rho,u,g)\in L^{\infty}(0,\infty;\dot{B}^{-\frac 1 2}_{2,\infty})\times L^{\infty}(0,\infty;\dot{B}^{-\frac 1 2}_{2,\infty})\times L^{\infty}(0,\infty;\dot{B}^{-\frac 1 2}_{2,\infty}(\mathcal{L}^2)).
\end{align*}
Taking advantage of Propositions \ref{prop5} with $\alpha=\sigma=\frac 1 2$ and $\beta=\frac 5 8$, we deduce that
\begin{align*}
E_{\eta}(t)\leq C(1+t)^{-\frac 5 8}~~~~and~~~~E^1_{\eta}(t)\leq C(1+t)^{-\frac 5 8-1}.
\end{align*}
According to Proposition \ref{prop2} and Lemma \ref{prop4} with $\sigma=1$ and $\alpha=\frac 5 8$, we obtain
\begin{align*}
(\rho,u,g)\in L^{\infty}(0,\infty;\dot{B}^{-1}_{2,\infty})\times L^{\infty}(0,\infty;\dot{B}^{-1}_{2,\infty})\times L^{\infty}(0,\infty;\dot{B}^{-1}_{2,\infty}(\mathcal{L}^2)).
\end{align*}
Using Propositions \ref{prop5} again with $\alpha=\frac 5 8$ and $\sigma=\beta=1$, we verify that
\begin{align*}
E_{\eta}(t)\leq C(1+t)^{-1}~~~~and~~~~E^1_{\eta}(t)\leq C(1+t)^{-2}.
\end{align*}

To get the faster decay rate for $g$ in $L^2$, we need the following standard energy estimation to \eqref{eq1}:
\begin{align*}
\frac {1} {2}\frac {d} {dt} \|g\|^2_{L^2(\mathcal{L}^{2})}+\|\nabla_R g\|^2_{L^2(\mathcal{L}^{2})}\lesssim -\int_{\mathbb{R}^{d}}\nabla u:\tau dx+\|\nabla u\|_{L^\infty}\|g\|_{L^2(\mathcal{L}^{2})}\|\nabla_R g\|_{L^2(\mathcal{L}^{2})}.
\end{align*}
By Lemma \ref{Lemma1}, Lemma \ref{Lemma2}, we have
\begin{align*}
\int_{\mathbb{R}^{d}}\nabla u:\tau dx\leq \delta\|\nabla_R g\|^2_{L^2(\mathcal{L}^{2})}+C_\delta \|\nabla u\|^2_{L^2}.
\end{align*}
Using Lemma \ref{Lemma1} and Theorem \ref{th1}, we deduce that
\begin{align*}
\frac d {dt} \|g\|^2_{L^2(\mathcal{L}^2)}+\|g\|^2_{L^2(\mathcal{L}^2)}\leq C_\delta \|\nabla u\|^2_{L^2},
\end{align*}
which implies that
\begin{align*}
\|g\|^2_{L^2(\mathcal{L}^2)}
&\leq \|g_0\|^2_{L^2(\mathcal{L}^2)}e^{-t}+C\int_{0}^{t}e^{-(t-s)}\|\nabla u\|^2_{L^2}ds  \\
&\leq C(e^{-t}+\int_{0}^{t}e^{-(t-s)}(1+s)^{-2}ds)  \\
&\leq C(1+t)^{-2}.
\end{align*}
We thus complete the proof of Theorem \ref{th2}.
\hfill$\Box$

\begin{rema}
One can see that the decay rate for the $\dot{H}^1$-norm obtained in Propositions \ref{prop3} is not optimal. However, in the proof of Theorem \ref{th2}, we improve the decay rate to $(1+t)^{-1}$.
\end{rema}

\begin{rema}
In Theorem \ref{th2}, we only obtain the optimal decay rate with $d=2$. The decay rate for $d=1$ is an interesting problem. However, the technique in this paper fail to obtain the optimal decay rate when $d=1$. We are going to study about this problem in the future.
\end{rema}

\smallskip
\noindent\textbf{Acknowledgments} This work was
partially supported by the National Natural Science Foundation of China (No.11671407 and No.11701586), the Macao Science and Technology Development Fund (No. 098/2013/A3), and Guangdong Province of China Special Support Program (No. 8-2015),
and the key project of the Natural Science Foundation of Guangdong province (No. 2016A030311004).


\phantomsection
\addcontentsline{toc}{section}{\refname}
\bibliographystyle{abbrv} 
\bibliography{Feneref}

\end{document}